\documentclass[a4paper,11pt]{article}
\usepackage{amssymb}
\usepackage{amsmath}
\usepackage{amsthm}
\usepackage{graphicx}
\theoremstyle{definition}
\newtheorem{definition}{Definition}[section]
\newtheorem{theorem}[definition]{Theorem}
\newtheorem{proposition}[definition]{Proposition}
\newtheorem{observation}[definition]{Observation}
\newtheorem{lemma}[definition]{Lemma}
\newtheorem{corollary}[definition]{Corollary}

\newtheorem{claim}[definition]{Claim}

\usepackage{amsmath,amsfonts,url,amssymb, latexsym}
\usepackage{tikz}
\usepackage[english]{babel}
\usepackage{verbatim}
\date{}

\newcommand{\phie}{\phi^{\sss \exists}}
\newcommand{\phiu}{\phi^{\sss \forall}}
\newcommand{\mse}{m_{\sss \exists}}
\newcommand{\msu}{m_{\sss \forall}}

\newcommand{\tp}{{\rm tp}}
\newcommand{\UF}{\mbox{$\mbox{\rm UF}_1^{\sss =}$}}

\usepackage{pgffor}
\usepackage{ifthen}
\usepackage{enumerate}

\newcommand{\UFO}[3]{
\ifthenelse{\equal{#3}{111}}
{$\forall^#1_{TC}[#2]$}{$\forall^#1_{TC}[#2,#3]$}}

\newcommand{\mA}{\mathfrak{A}}

\renewcommand{\phi}{\varphi} 

\newcommand{\AAA}{\mbox{\large \boldmath $\alpha$}}

\newcommand{\FOt}{\mbox{$\mbox{\rm FO}^2$}}

\newcommand{\set}[1]{\{#1\}}

\newcommand{\str}[1]{{\mathfrak{#1}}}
\newcommand{\restr}{\!\!\restriction\!\!}

\newcommand{\N}{{\mathbb N}}

\newcommand{\sss}{\scriptscriptstyle}

{\begin{list}{{\rm (\roman{enumiii})}}{\usecounter{enumiii}
}}{\end{list}}

{\begin{list}{{\rm (\alph{enumii})}}{\usecounter{enumii}
}}{\end{list}}

\begin{document}



\title{Complexity and Expressivity of  Uniform One-Dimensional Fragment with Equality}
%
%
\author{Emanuel Kiero\'nski\footnote{Institute of Computer Science, University of Wroc\l aw, Poland}\, \, and Antti Kuusisto\footnote{University of Tampere Logic Group and Institute of Computer Science,
University of Wroc\l aw, Poland}}
\maketitle
%
%
%


\begin{abstract}
\noindent
Uniform one-dimensional fragment \UF{} is a formalism 
%
obtained from first-order logic by limiting 
quantification to applications of blocks of existential (universal)
quantifiers such that at most
one variable remains free in the quantified formula.
The fragment is closed under Boolean operations,
but additional restrictions (called uniformity conditions)
apply to combinations of atomic formulas with two or more variables.
The fragment can be seen as a canonical generalization of two-variable logic,
defined in order to be able to deal with relations of arbitrary arities.
The fragment was introduced recently, and it was shown that the satisfiability problem
of the equality-free fragment $\mathrm{UF}_1$ of $\UF{}$ is decidable.
In this article we establish that the satisfiability and finite satisfiability problems of 
$\UF{}$ are $\mathrm{NEXPTIME}$-complete. We also show that the corresponding
problems for the extension of $\UF{}$  with counting quantifiers
are undecidable.
In addition to decidability questions, we compare the expressivities of $\UF{}$ and two-variable logic
with counting quantifiers $\mathrm{FOC}^2$. 
We show that while the logics
are incomparable in general, $\UF{}$ is strictly contained in $\mathrm{FOC}^2$
when attention is restricted to vocabularies with the arity bound two.

\end{abstract}

\section{Introduction}

\newcommand{\UFC}{\mbox{$\mbox{\rm UFC}_1^=$}}

%
%

Two-variable logic $\mathrm{FO}^2$ was introduced by Henkin
in \cite{IEEEonedimensional:henkin}
and proved decidable in \cite{IEEEonedimensional:mortimer} by Mortimer.
%
%
%
The satisfiability and finite satisfiability problems of 
$\mathrm{FO}^2$ were shown to
be $\mathrm{NEXPTIME}$-complete in \cite{GKV97}.
The extension of two-variable logic with counting quantifiers, $\mathrm{FOC}^2$,
was proved decidable in \cite{IEEEonedimensional:gradel}, \cite{PST97}.
It was subsequently shown to be
$\mathrm{NEXPTIME}$-complete in \cite{IEEEonedimensional:pratth}.
Research on extensions and variants of
two-variable logic is \emph{currently very active}.
Recent research efforts have mainly concerned decidability
and complexity issues
in restriction to particular classes of 
structures, and also questions
related to different built-in features and operators that 
increase the expressivity of the base language.
Recent articles in the field include for example
\cite{IEEEonedimensional:kieronski},
\cite{IEEEonedimensional:charatonic},
%
%
\cite{IEEEonedimensional:kieronskitendera},
%
%
\cite{IEEEonedimensional:tendera},
and several others.
%

%
%

%
%

%
%
%
Typical systems of modal logic are contained  in two-variable logic, or some variant of it,
and hence investigations on two-variable logics have
direct implications on various fields of computer science, including
verification of software and hardware, distributed systems,
knowledge representation and artificial intelligence.
However, two-variable logics do not cope well
with relations of arities greater than two, and therefore
the \emph{scope of related research is significantly restricted}.
In database theory
contexts, for example, two-variable logics as such are usually not
directly applicable due to the \emph{severe arity-related limitations}.
The recent article \cite{HK14} introduces the \emph{uniform one-dimensional fragment}, 
$\UF{}$, which is a natural generalization of $\mathrm{FO}^2$ to 
contexts with relations of arbitrary arities.
The logic $\UF{}$ is a fragment of first-order logic obtained by restricting
quantification to blocks of existential (universal) quantifiers
that \emph{leave
at most one free variable} in the resulting formula.
Additionally, a \emph{uniformity condition} applies to the use of
atomic formulas: if $n,k\geq 2$, then a Boolean combination of
atoms $R(x_1,...,x_k)$ and $S(y_1,...,y_n)$
is allowed only if $\{x_1,...,x_k\} = \{y_1,...,y_n\}$.
Boolean combinations of formulas with at most one
free variable can be formed freely,
and the use of equality is unrestricted.
%

%
%
%
%

%
It is established in \cite{HK14} that already the equality-free fragment
$\mathrm{UF}_1$ of $\UF{}$ can
define properties not expressible in $\mathrm{FOC}^2$ and also properties not expressible in
the recently introduced \emph{guarded negation fragment}  \cite{IEEEonedimensional:barany},
which significantly generalizes the \emph{guarded fragment} \cite{IEEEonedimensional:andreka}.
%
%
The article \cite{HK14} also shows, inter alia,
that the equality-free logic $\mathrm{UF}_1$ is decidable, and furthermore, that
minor modifications to the syntax of
$\mathrm{UF}_1$ lead to undecidable formalisms. Namely, the
\emph{uniform two-dimensional} and \emph{non-uniform one-dimensional}
fragments are shown undecidable.
In this article we establish that the satisfiability and finite satisfiability
problems of the uniform one-dimensional fragment
with equality ($\UF{}$) are $\mathrm{NEXPTIME}$-complete.
These results are obtained by appropriately generalizing and modifying
the construction in \cite{GKV97} that provides small models
for satisfiable \FOt{}-formulas in \emph{Scott normal form}.
The $\mathrm{NEXPTIME}$-completeness of $\mathrm{FOC}^2$ raises the natural question whether
the extension $\UFC$ of $\UF{}$ with \emph{counting quantifiers}
remains decidable. We answer this question in the 
negative by showing that the satisfiability and
finite satisifiability problems of $\UFC$ are
complete for $\Pi_1^0$ and $\Sigma_1^0$,
respectively. These results are established by tiling
arguments that make an appropriate use of a \emph{ternary} relation,
together with the usual unary relations commonly employed in
similar undecidability proofs.
We also study the expressivity of $\UF{}$.
We establish that while $\UF{}$ and $\mathrm{FOC}^2$ 
are incomparable in expressivity in general,
in restriction to vocabularies with the
arity bound two, we have $\UF{} < \mathrm{FOC}^2$.
The uniform one-dimensional fragment $\UF{}$ canonically extends $\mathrm{FO}^2$,
and in fact the equality-free fragments of $\UF{}$ and $\mathrm{FO}^2$
coincide when attention is limited to binary vocabularies.
We believe that $\UF{}$ is an interesting fragment
that can be used in order to 
\emph{extend the scope of research on
two-variable logics to the realm involving relations of arbitrary arities.}
\section{Preliminaries}

Let $m$ and $n\geq m$ be integers. We let $[m,n]$ denote
the set of integers $i$ such that $m\leq i \leq n$.
If $\varphi$ and $\psi$ are first-order formulas,
then $\varphi \equiv \psi$ indicates that the formulas are equivalent.
If $\mathcal{L}$ and $\mathcal{L}'$ are fragments of first-order logic,
we write $\mathcal{L}\leq \mathcal{L}'$ to indicate that for every sentence
of $\mathcal{L}$, there exists an equivalent sentence of $\mathcal{L}'$.
We let $\mathrm{VAR} := \{\, v_i\ |\ i\in\mathbb{N}\ \}$
denote the set of first-order variable symbols.
We mostly use \emph{metavariables} $x,y,z, x_i, y_i, z_i$, etc.,
in order to refer to symbols in $\mathrm{VAR}$.
Notice that for example $x_1$ and $x_2$ may denote
the same variable in $\mathrm{VAR}$, while
$v_1$ and $v_2$ are necessarily
different variables.
The set of free variables of a formula $\psi$ is denoted by $\mathit{free}(\psi)$.
%


%
%
%

%
Let $X = \{x_1,...,x_n\}$ be a finite set of
variable symbols.
%
%
Let $R$ be a
$k$-ary
relation symbol.
%
%
An atomic formula $R(x_{i_1},...,x_{i_k})$ is called an \emph{$X$-atom}
if $\{x_{i_1},...,x_{i_k}\} = X$.
A finite set of $X$-atoms is an \emph{$X$-uniform set}.
When $X$ is irrelevant or known from the context, we may
simply talk about a \emph{uniform set}.
%
%
%
For example, if $x,y,z$ are distinct variables,
then $\{T(x,y), S(y,x)\}$ and $\{R(x,x,y),R(y,y,x), S(y,x)\}$ are uniform sets,
while $\{R(x,y,z),R(x,y,y)\}$ and $\{S(x,y), x=y\}$ are not (uniform sets
are not allowed to contain equality atoms).
The empty set is an $X$-uniform set for every finite subset of $\mathrm{VAR}$,
including $\emptyset$.
Let $\mathbb{Z}_+$ denote the set of positive integers.
Let $\mathcal{V}$ denote a \emph{complete
relational vocabulary}, i.e., $\mathcal{V} := \bigcup_{k\, \in\, \mathbb{Z}_+} \tau_k$,
where $\tau_k$ denotes a countably infinite
set of $k$-ary relation symbols.
Every vocabulary $\tau$ we consider below
is assumed to be a subset of $\mathcal{V}$.
%
%
%
%
%
A \emph{$k$-ary $\tau$-atom} is an atomic $\tau$-formula $\psi$
such that $|\mathit{free}(\psi)| = k$. For example, if $P\in\tau$
is a unary and $R\in\tau$ a binary symbol, then 
$P(x)$, $x=x$, $R(x,x)$ are unary $\tau$-atoms, and
$R(v_1,v_2)$, $v_1=v_2$ are binary $\tau$-atoms.
If $\tau$ is known 
form the context or irrelevant, we may simply
talk about $k$-ary atoms.
Let $\tau\subseteq \mathcal{V}$.
The set $\UF{}(\tau)$, or the set of $\tau$-formulas of
the uniform one-dimensional fragment, is
the smallest set $\mathcal{F}$ satisfying the following conditions.
\begin{enumerate}
\item
Every unary $\tau$-atom is in $\mathcal{F}$.
Also $\bot,\top\in \mathcal{F}$.
\item
Every identity atom $x=y$ is in $\mathcal{F}$.
\item
If $\varphi\in \mathcal{F}$, then $\neg\varphi\in\mathcal{F}$.
If $\varphi_1,\varphi_2\in \mathcal{F}$,
then $(\varphi_1\wedge\varphi_2)\in \mathcal{F}$.
\item
Let $X = \{x_0,...,x_k\}\subseteq\mathrm{VAR}$.
Let $U$ be a finite set of
formulas $\psi\in\mathcal{F}$ whose
free variables are in $X$.
Let $V\subseteq X$.
Let $F$ be a $V$-uniform set of $\tau$-atoms.
Let $\varphi$ be any Boolean combination of formulas in $U\cup F$.
Then $\exists x_1...\exists x_k\, \varphi\, \in \mathcal{F}$ and
$\exists x_0...\exists x_k\, \varphi\, \in \mathcal{F}$.
%
%
%
\end{enumerate}
%
%
%
%
%
%
Let $\UF{}$ denote the set $\UF{}(\mathcal{V})$.
%
%
%

%
Let $\overline{x}$ denote a tuple of variables, and
let $\chi := \exists \overline{x}\, \varphi$ be a $\UF{}$-formula
formed by using the rule $4$ above. Assume that $\varphi$ is quantifier-free.
Then we call $\varphi$ a \emph{$\UF{}$-matrix}.
If $\varphi$ does not contain $k$-ary atoms for any $k\geq 2$,
with the possible exception of equality atoms $x= y$, 
then we define $S_{\varphi} :=\emptyset$.
Otherwise we define $S_{\varphi}$ to be 
the set $V$ used in the construction of $\chi$ (see rule 4).
The set $S_{\varphi}$ is the set of \emph{live variables} of $\varphi$.
Let $\psi(x_0, \ldots, x_k)$ be a $\UF{}$-matrix, where
$(x_0,...,x_k)$ enumerates the variables of $\psi$.
%
%
Let $\str{A}$ be a structure. Let $a_0, \ldots, a_k \in A$,
where $A$ is the domain of $\str{A}$.
We let $\mbox{live}\bigl(\psi(x_0, \ldots, x_k)[a_0, \ldots, a_k]\bigr)$
denote the set $T\subseteq \set{a_0, \ldots, a_k}$
such that $a_i\in T$ iff $x_i$ is a live variable of  $\psi(x_0, \ldots, x_k)$.
We may write $\mbox{live}\bigl(\psi[a_0, \ldots, a_k]\bigr)$
instead of $\mbox{live}\bigl(\psi(x_0, \ldots, x_k)[a_0, \ldots, a_k]\bigr)$ when no 
confusion can arise. Notice that since the elements $a_i$ are not 
required to be distinct, it is possible that
$|\mbox{live}(\psi[a_0, \ldots, a_k])|$ is smaller than the number of live variables in $\psi$.
%


%
%
%

%
Let $\tau\subseteq\mathcal{V}$ be a finite vocabulary.
A \emph{$1$-type} over the vocabulary $\tau$ is a maximal
satisfiable set of literals (atoms and negated atoms) over $\tau$ with the variable $v_1$.
The set of all $1$-types over $\tau$ is denoted by $\AAA[\tau]$,
or just by $\AAA$ when $\tau$ is clear from the context. 
%
We identify $1$-types $\alpha$ and conjunctions
$\bigwedge\alpha$.
%
%
%
A {\em $k$-table} over $\tau$ is a
maximal satisfiable set of $\{v_1,...,v_k\}$-atoms and 
negated $\{v_1,...,v_k\}$-atoms over $\tau$.
Recall that a $\{v_1,...,v_k\}$-atom must contain exactly
all the variables in $\{v_1,...,v_k\}$, and note that a $2$-table
does not contain equality formulas or negated equality formulas.
We identify $k$-tables $\beta$ and
conjunctions $\bigwedge\beta$.
Let $\str{A}$ be a $\tau$-structure, and let $a \in A$.
Let $\alpha$ be a $1$-type over $\tau$.
We say that $a$ \emph{realizes} $\alpha$ if $\alpha$ is the unique $1$-type
such that $\str{A} \models \alpha[a]$.  We let $\tp_\str{A}(a)$ denote the $1$-type realized by $a$. Similarly, for
\emph{distinct} elements $a_1, \ldots, a_k \in A$, we let $\mathrm{tb}_\str{A}(a_1, \ldots, a_k)$ denote the unique
$k$-table \emph{realized} by the tuple $(a_1,\ldots, a_k)$, i.e.,~the $k$-table $\beta(v_1,...,v_k)$
such that $\str{A} \models \beta[a_1,\ldots, a_k]$. 
Note that we have $\tp_\str{A}(a) \equiv \mathrm{tb}_\str{A}(a)$ for every $a \in A$.
Let $m$ be the maximum arity of symbols in $\tau$.
We observe that to fully define
a $\tau$-structure $\str{A}$ over a known domain $A$,
it is sufficient to consider each set $B \subseteq A$, $|B| \leq m$,
and first 
choose an enumeration $(b_1, \ldots, b_{|B|})$ of the elements of $B$,
and then
specify $\mathrm{tb}_\str{A}(b_1, \ldots, b_{|B|})$.
%
%
%
%
%
\begin{observation}\label{o:satisfaction}
Let $\psi(x_1, ..., x_k)$ be a $\UF$-matrix,
where $(x_1,...,x_k)$ enumerates the variables in $\psi$.
Let $\str{A}$ be a $\tau$-structure, where $\tau$
is the set of relation symbols in $\psi$.
Let $a_1, ..., a_k\in A$ be a sequence of (not necessarily distinct) elements.
Whether or not $\str{A} \models \psi[a_1, ..., a_k]$ holds,
depends only on
(i) the $1$-types of the elements $a_i$,
(ii) the list of pairs $(a_i,a_j)$ such that $a_i = a_j$,
and (iii) the table $\mathrm{tb}_\str{A}(b_1, ..., b_l)$, where $(b_1,..., b_l)$ is an
arbitrary enumeration of $\mbox{live}(\psi[a_1, ..., a_k])$.
\end{observation}

\section{Complexity of $\UF{}$}\label{decidabilitysection}


%
We now introduce a normal form for \UF{} inspired by the \emph{Scott normal form}  for \FOt{} \cite{Sco62}.
We say that a \UF{}-formula $\phi$ is in \emph{generalized Scott normal form} if 
\begin{equation}\label{normalform}
\phi=\bigwedge_{1\, \le\, i\, \le\, \mse}\hspace{-10pt} \forall x \exists y_1 \ldots y_{k_i}
\phie_i(x, y_1, \ldots, y_{k_i}) \wedge \bigwedge_{1\, \le\,  i\,  \le\, \msu}\hspace{-10pt}
\forall x_1 \ldots x_{l_i} \phiu_i(x_1, \ldots, x_{l_i}),
\end{equation}
where formulas $\phie_i$ and $\phiu_i$ are quantifier-free \UF{}-matrices.


\begin{proposition}
Each \UF{}-formula $\phi$ translates in polynomial time to a
\UF{}-formula $\phi'$ in  generalized Scott normal form (over a signature extended by some
fresh unary symbols) such that $\phi$ and $\phi'$ are satisfiable over the same domains.
\end{proposition}

\begin{proof}
A  simple adaptation of the well-known
translation given, e.g., in \cite{Ebbinghaus}.
\end{proof}

Let $\varphi$ be a  $\UF$-formula in generalized Scott-normal form given in Equation \ref{normalform}.
Assume $\str{A}\models\varphi$.
We will build a small $\tau$-model $\str{A'}$ of $\phi$,
where $\tau$ is the set of relation symbols in $\varphi$. Our construction
modifies and generalizes the construction of a small model
for a satisfiable \FOt{}-formula in Scott normal from \cite{GKV97}.
Let $a \in A$ and $b_1, \ldots, b_{k_i}\in A$ be elements such that $\str{A} \models \phie_i[a, b_1, \ldots, b_{k_i}]$.
We say that the structure $\str{B} := \str{A} \restr \set{a, b_1, \ldots, b_{k_i}}$
is a \emph{witness structure} for $a$ and $\phie_i$. The substructure of $\str{B}$ restricted to the elements in $\mbox{live}(\phie_i[a, b_1, \ldots, b_{k_i}])$ 
is called the \emph{live part} of $\str{B}$. If the live part of $\str{B}$ does not contain $a$, then the live part is called \emph{free}. Note that $|B|$ may be smaller
than $k_i+1$ 
(this may be even imposed by the use of equalities). Also, $a$ may be a member of the live part of  $\str{B}$ even if the variable
$x$ is not a live variable of $\phie_i$.

\medskip\noindent
{\em The court.}
Let $n$ be the width of $\phi$, i.e.,~$n=\max (\set{k_i+1}_{1 \le i \le \mse} \cup \set{l_i}_{1 \le i \le \msu})$. 
We assume, w.l.o.g., that $n\geq 2$.
A $1$-type $\alpha$ realized in $\str{A}$ is \emph{royal} if it is realized at most $n-1$ times in $\str{A}$. The points in $A$ that realize a
royal $1$-type are called \emph{kings}. Let $K$ be the set of all kings in $\str{A}$. Clearly $|K| \le (n-1) |\AAA|$
%
%

We then define a set $D\subseteq A$.
For each pair $(\alpha$, $\phie_i)$, where $\alpha$ is a $1$-type realized in $\str{A}$, if it is possible,
select an element $a \in A$ that realizes the $1$-type $\alpha$ such that 
there exists a witness structure $\str{B}_{\alpha,i}$ for $a$ and $\phie_i$ whose live part $\bar{\str{B}}_{\alpha, i}$ is free.
Add the elements of $\bar{B}_{\alpha,i}$ to $D$. 
Since we add at most $n-1$ elements for each pair $(\alpha$, $\phie_i)$, the total size of $D$ is
bounded by $(n-1)\mse |\AAA|$.

For each $a \in K \cup D$ and each $\phie_i$, select a witness structure in $\str{A}$ and
let $C_{a,i}$ denote its universe. 
Define $C := K \cup D \cup \bigcup_{a \in K \cup D, 1 \le i \le \mse} C_{a,i}$. We call
$\str{C} := \str{A}\upharpoonright C$ the \emph{court} of $\str{A}$. Note that
$|C| \le n|K \cup D| \le n((n-1)+(n-1)\mse)|\AAA|$. 
We have $|C| \le 2 |\phi|^3 2^{|\phi|}$. 

\medskip \noindent
{\em Universe.}
The court $\str{C}$ of $\str{A}$ will be a substructure of $\str{A}'$. 
The remaining part of the universe of $\str{A}'$
consists of three fresh disjoint sets $E, F, G$. 
Each of them contains $\mse + n$ elements of type $\alpha$ for each non-royal $\alpha$ 
realized in $\str{A}$.
The $i$-th element of type $\alpha$ ($1 \le i \le \mse + n$) in $E$ (resp.~$F$, $G$) is denoted
$e_{\alpha, i}$ (resp.~$f_{\alpha, i}$~, $g_{\alpha, i}$).
The size of each set $E$, $F$, $G$ is bounded by $(n+\mse)|\AAA| \le 2|\phi|2^{|\phi|}$. Thus
the total size of $|A'|$ is bounded by $8 |\phi|^3 2^{|\phi|}$, which is exponential in $|\phi|$.
\medskip \noindent
{\em Witnesses.} Our next aim is to provide witness structures for each element of $a \in A' \setminus (K\cup D)$
and each $\phie_j$.
We will choose elements in $A'$ which will form the universe (say, of size $s$) of the live part of a witness structure for $a$ and $\phie_j$
and define the $s$-table on these elements. The remaining elements of the witness structure
(elements not in the live part) will then be very easily found in $A'$.

Let $a' \in A' \setminus (K\cup D)$. We find a
\emph{pattern element} $a \in A$ of $a'$ as follows.
If $a' \in C$, then the pattern element is $a'$ itself.
If $a' \in E \cup F \cup G$, then we let an arbitrary $a \in A$ such
that $\tp_\str{A}(a)=\tp_\str{A'}(a')$ be the pattern element of $a'$.
For each $1 \le j \le \mse$, we find a witness structure $\str{B}_{a,j}$ for
$a$ and $\phie_j$, and let $\bar{\str{B}}_{a,j}$ be its live part. If this live part is free, then
there is nothing to do; an appropriate live part for the witness
structure of $a'$ and $\phie_j$ already exists in 
$\str{D} := \str{C}\upharpoonright D$. Otherwise, let $r_1, \ldots, r_k$ be the kings included in
$\bar{\str{B}}_{a,j}$ (possibly $k=0$), and let $a, b_1, \ldots, b_l$ be the non-royal
elements of $\bar{\str{B}}_{a,j}$ (possibly $l=0$). Let $\alpha_i$ be the $1$-type of $b_i$
($1 \le i \le l$). We consider the following cases.

\medskip\noindent
\emph{Case 1.} If $l=0$ and $a' \in C$, then there is nothing to do; $a'$ forms the live part of the desired witness
structure together with some elements in $K$.

\medskip\noindent
\emph{Case 2.}
If $l=0$ and $a' \not\in C$, then we 
set $\mathrm{tb}_{\str{A}'}(a',r_1, \ldots, r_k) := \mathrm{tb}_{\str{A}}(a,r_1, \ldots, r_k)$.

\medskip\noindent
\emph{Case 3.} If $l>0$ and $a' \in E$,
then we define $b_1' :=f_{\alpha_1, j}$ and choose $b_2', \ldots, b_l'$ to be
distinct elements of types $\alpha_2, \ldots, \alpha_l$ from 
$S := \{ f_{\alpha, s}: \mse +1 \le s \le \mse +n, \alpha \mbox{ non-royal} \}$. 
This is possible since $l < n$ and $S$ contains $n$ realizations of each non-royal $1$-type.
We set \vspace{-5.5pt}
\begin{center}
$\mathrm{tb}_{\str{A}'}(a',r_1, \ldots, r_k, b_1', \ldots, b_l') :=
\mathrm{tb}_{\str{A}}(a,r_1, \ldots, r_k, b_1, \ldots, b_l).$
\end{center}
\vspace{-8pt}
\medskip\noindent
\emph{Case 4.} 
If $l>0$ and $a' \in F$ (resp.~$a' \in G \cup (C \setminus (K\cup D))$), then we proceed as in the previous case, but we take
the elements $b_i'$ from $G$ (resp.~$E$).

\medskip

\begin{claim}
The described  procedure of providing live parts of witness structures can be executed without conflicts.
\end{claim}
\begin{proof}
Consider two tuples of elements $t'=a', r'_1, \ldots, r'_{k'}, b'_1, \ldots, b'_{l'}$ 
and $t''=a'', r''_1, \ldots, r''_{k''}, b''_1, \ldots, b''_{l''}$ that form in the described process the live parts
of witness structures for, respectively, $a', \phie_i$ and $a'', \phie_j$.

If $a' \not= a''$, then the two tuples differ, since our strategy of using the sets $E$, $F$, $G$
ensures that $a'$ cannot be a member of $t''$ and $a''$ cannot be a member of $t'$.

If $a'=a''$, $i \not=j$ and $l' >0$, then the two tuples differ since $b'_1$ cannot be a member of
$t''$, due to the following reason. Since $b_1'$ is selected to be the $i$-th realization of some $\alpha$ from $E$, $F$, or $G$, it must differ from $b''_1$, which is the $j$-th 
realization of some $\alpha$. Furthermore, $b_i'$ must differ from each $b_m''$, $m>1$, since
each of them is the $s$-th realization of some $\alpha$ for some $s>\mse$,
while $i\leq \mse$. The case where
$a'=a''$, $i \not=j$ and $l'' >0$
is symmetric.

If $a'=a''$, $i \not=j$ and $l'=l''=0$, then it is possible that
the two tuples are identical, but in this case they contain only $a$ and some kings $r'_1, \ldots, r'_{k'}$,
and even though $\mathrm{tb}_{\str{A}'}(a', r'_1, \ldots, r'_{k'})$ is defined twice, it is
done precisely the same way both times
%
(setting $\mathrm{tb}_{\str{A}'}(a', r'_1, \ldots, r'_{k'}) := \mathrm{tb}_{\str{A}'}(c, r'_1, \ldots, r'_{k'})$ both times,
where $c$ is the pattern element of $a'$).
%
%
\end{proof}

It is probably worth commenting why we prepared free live parts of witness structures in $\str{D}$ instead of building them
using elements of $E \cup F \cup G$ in a ``regular'' way.
One of the problematic situations arises, e.g., when an element $a'$ from, say, $E$ builds the live part
of its witness structure for some $\phie_i$ using an element $b' \in F$ and some kings $r_1, \ldots, r_k$. In this case $\mathrm{tb}_{\str{A}'}(b', r_1, \ldots, r_k)$
is defined. However, it may happen that $b'$ needs to form the live part of its witness structure for some $\phie_j$ using precisely the elements $b', r_1, \ldots, r_k$, which
can lead to a conflict.
 
\medskip \noindent
{\em Completion.}
Let $a'_1, \ldots, a'_k$ ($a'_i \not=a'_j$ for $i \not=j$, $1 < k \le n$) be elements in $A'$ such that
the table $\mathrm{tb}_{\str{A}'}(a'_1, \ldots, a'_k)$
has not yet been defined. Select distinct elements $a_1, \ldots, a_k$ of $\str{A}$ such that
$\tp_\str{A}(a_i)=\tp_{\str{A}'}(a'_i)$ ($1 \le i \le k$). This is always possible due to our strategy of
not introducing extra kings. Set $\mathrm{tb}_{\str{A}'}(a'_1, \ldots, a'_k) := \mathrm{tb}_{\str{A}}(a_1, \ldots, a_k)$.

\medskip

\begin{claim}
$\str{A}' \models \phi$.
\end{claim}
\begin{proof}
Let us first argue that all existential requirements  are fulfilled.
Consider an arbitrary element $a' \in A$ and  a conjunct $\forall x \exists y_1 \ldots y_{k_i} \phie_i(x, y_1, \ldots, y_{k_i})$ of $\phi$. If $a' \in K \cup D$, then the witness structure  
for $a'$ and $\phie_i$ is provided in $\str{C}$ by the elements of $C_{a',i}$.
If $a' \in (C \setminus (K\cup D)) \cup E \cup F \cup G$, then
the live part of a  witness structure is either secured in $\str{D}$ or according to 
one of the four cases in the step \emph{Witnesses}. 
We consider the case where the live
part is not in $\str{D}$ (the arguments involving $\str{D}$ are similar).
Assume that the pattern element for $a'$ is $a$ and the witness structure for $a$ and $\phie_i$ is
$\str{B}_{a,i}$, where $B_{a,i}=\set{a, r_1, \ldots, r_{k'}, b_1, \ldots, b_{l'}}$.
Here the elements $r_i$ are kings
and the elements $b_i$ non-kings.
Assume that the universe of the live part of $\str{B}_{a,i}$ is $\bar{B}_{a,i}=\set{a, r_1, \ldots, r_{k}, b_1, \ldots, b_{l}}$, $l \le l'$, $k \le k'$.
(Note that in the case we are considering, $a$ is necessarily a member of $\bar{B}_{a,i}$). A witness structure for $a'$ can now be formed using the elements $a',r_1, \ldots, r_{k'}$,
the elements $b_1', \ldots, b_l'$ (if any) chosen
according to \emph{Case 3} or $\emph{Case 4}$, and elements $b'_{l+1}, \ldots, b'_{l'}$ from, say, the set $E$, 
whose $1$-types are equal to the $1$-types of $b_{l+1}, \ldots, b_{l'}$. Note that we always have sufficiently many elements with appropriate $1$-types in $E$.
Due to Observation \ref{o:satisfaction}, the substructure $\str{A}' \restr \set{a', r_1, \ldots, r_{k'}, b'_1, \ldots, b'_{l'}}$ is indeed a an appropriate witness structure.

Let us then argue that the universal conjuncts of $\phi$ are satisfied. Consider a
conjunct $\forall x_1 \ldots x_k \phiu_i(x_1, \ldots, x_k)$ and a tuple of 
(not necessarily distinct) elements $a'_1, \ldots, a_k'$ from $A'$. 
We must show that $\str{A}' \models \phiu_i[a_1', \ldots, a_k']$. 
Let $b_1', \ldots, b_l'$ be the elements of $\mbox{live}(\phiu_i[a'_1, \ldots, a_k'])$.
There are three different ways
in which $\mathrm{tb}_{\str{A}'}(b_1', \ldots, b_l')$
can become defined: if $l=1$, then $\mathrm{tb}_{\str{A}'}(b_1', \ldots, b_l')$
is just a $1$-type, and the table is fixed in step \emph{Universe}, and if $l>1$, then the table
is fixed either according to step \emph{Witnesses} or step \emph{Completion}. In each of the cases there are
distinct elements $b_1, \ldots, b_l \in A$ such that $tp_\str{A}(b_i)=\tp_{\str{A}'}(b'_i)$ for $1 \le i \le l$ and
$\mathrm{tb}_{\str{A}'}(b_1', \ldots, b_l')=\mathrm{tb}_{\str{A}}(b_1, \ldots, b_l)$.
Let $b_1', \ldots, b_l', c'_1, \ldots, c'_s$ be a list of all distinct elements of $\set{a'_1, \ldots, a_k'}$.
It is readily verified that the tuple $b_1, \ldots, b_l$ can be extended by elements $c_1, \ldots, c_s$ to a tuple
of distinct elements of $A$ with $\tp_\str{A}(c_i)=\tp_\str{A'}(c_i')$ for $1 \le i \le s$. 
Thus we can form a sequence of (not necessarily distinct) elements $a_1, \ldots, a_k$ of $A$ corresponding
to $a'_1, \ldots, a'_k$ in such a way that $a'_i=b'_j\Leftrightarrow a_i=b_j$ and
$a'_i=c'_j\Leftrightarrow a_i=c_j$. Obviously $\str{A} \models \phiu_i[a_1, \ldots, a_k]$. By Observation \ref{o:satisfaction},
it follows that $\str{A}' \models \phiu_i[a_1', \ldots, a_k']$. 
\end{proof}
We have  proved the following theorem.

\begin{theorem}
\UF{} has the finite model property. Moreover, every satisfiable \UF{} formula
$\phi$ has a model whose size is bounded exponentially in $|\phi|$. 
\end{theorem}

We can now prove the following theorem.
\begin{theorem}\label{complexitytheorem}
The satisfiability problem (=\hspace{0.6mm}finite satisfiability problem) for \UF{} is $\mathrm{NEXPTIME}$-complete.
\end{theorem}
\begin{proof}
\noindent
The lower bound follows from the $\mathrm{NEXPTIME}$ lower bound for \FOt{} from \cite{Lew80}.
For the upper bound,
we translate the input formula $\phi$ to an equisatisfiable formula
$\phi'$ in generalized Scott normal form.
Suppose $\phi'$ is satisfiable.
We guess an exponentially bounded model $\str{A}$ of $\phi'$ (note that not only the universe of
$\str{A}$ is bounded exponentially, but also the description
of $\str{A}$, since we are dealing only with at most $|\phi'|$ relations of arity at most $|\phi'|$),
and verify that it is indeed a model. The last task can be carried out in an exhaustive way: for each $a \in A$ and
each $\phie_i$, guess which elements form a witness structure and check that
they indeed form a required witness structure; for each conjunct $\forall x_1, \ldots, x_k \phiu_i(x_1, \ldots, x_k)$,
enumerate all tuples $a_1, \ldots, a_k$ of $A$ and check that $\str{A} \models \phiu_i[a_1, \ldots, a_k]$.
\end{proof}


\section{Expressivity}\label{expressivitysection}
In this section we compare the expressivity of \UF{} with
the expressivities of \FOt{} and $\mathrm{FOC}^2$.
Clearly \UF{} contains \FOt{}, and it is not hard to see that
the inclusion is strict; equalities can be used freely in $\UF{}$,
and for example the property that there are precisely
two elements in a unary relation $P$ is expressible in
$\UF{}$ but not in $\mathrm{FO}^2$.
The expressivities of $\UF{}$ and $\mathrm{FOC}^2$ are
related as follows.
%
%

\begin{theorem}\label{generalexpressivity}
$\UF{}$ and $\mathrm{FOC}^2$ are incomparable in expressivity.
\end{theorem}
\begin{proof}
It is straighforward to establish that $\mathrm{FOC}^2$ cannot express
the $\UF{}$-sentence
$\exists x \exists y \exists z R(x,y,z),$
and therefore $\UF{}\not\leq \mathrm{FOC}^2$.
To show that $\mathrm{FOC}^2\not\leq\UF{}$,
let $R$ be a binary relation symbol
and consider models over the signature $\{R\}$.
We claim that $\UF{}$ cannot express the $\mathrm{FOC}^2$-definable
condition that the in-degree (w.r.t. the relation $R$) at every node is at most one.  
Assume $\varphi(R)$ is a $\UF{}$-formula that
defines the condition. Consider the conjunction
$\varphi(R)\ \wedge\ \forall x\exists y R(x,y)\ \wedge\ \exists x \forall y\neg R(y,x).$
It is easy to see that this formula does not have a finite model,
and thereby the assumption that $\UF{}$ can express $\varphi(R)$ is false.
%
%
\end{proof}
%
%
%

The rest of this section is devoted to the scenario in which
the signature contains only unary and binary relation symbols.
We will show that in such a case the expressivity of \UF{} lies strictly
between \FOt{} and $\mathrm{FOC}^2$.

%
%

%
%
%
%

%
Let $\tau$ be a finite relational vocabulary.
Let \emph{$\beta$} be a $2$-table over $\tau$,
and let $x$ and $y$ be \emph{distinct} variables.
Let $S$ be the set of atoms obtained from $\beta$ by replacing 
all occurrences of the variables $v_1$ and $v_2$ in $\beta$ by $x$ and $y$,
respectively. We call $S$ a \emph{binary $\tau$-diagram} in the variables $(x,y)$,
and denote it by $\beta(x,y)$.
We identify binary diagrams and conjunctions over them.
A \emph{binary $\tau$-arrow} in the variables $x,y$
is an atomic formula $R(x,y)$ (or $R(y,x)$), where $R\in\tau$.
Notice that neither equality statements nor atoms of the form $R(x,x),R(y,y)$ are binary $\tau$-arrows.
It is straightforward to show that if $\varphi$ is a Boolean
combination of binary $\tau$-arrows in the variables $x,y$, then $\varphi$ is
equivalent to the disjunction of $\tau$-diagrams $\beta(x,y)$ that entail $\varphi$,
i.e., $\beta(x,y)\models \varphi$.
Notice that $\bigvee\emptyset = \bot$ is a legitimate disjunction of diagrams.
Let $\{x_0,...,x_k\}$ be a (possibly empty) set of distinct variables.
%
%
%
%
An \emph{identity literal} over $\{x_0,...,x_k\}$ is a
formula of the type $x_i = x_j$ or $\neg x_i = x_j$,
where $i,j\in [0,k]$. An identity literal 
is \emph{non-trivial} if the variables in it are different.
An \emph{identity profile} over $\{x_0,...,x_k\}$, or a \emph{$\{x_0,...,x_k\}$-profile}, is a
maximal satisfiable set of non-trivial identity literals over $\{x_0,...,x_k\}$.
We identify identity profiles 
and conjunctions over them.
We let $\mathit{diff}(x_0,...,x_k)$
denote the conjunction of inequalities $x_i\not=x_j$, where $i,j\in [0,k]$, $i\not=j$.
An identity profile is a  
\emph{discriminate profile} if it is 
the formula $\mathit{diff}(x_0,...,x_k)$ for some 
set $\{x_0,...,x_k\}$ of distinct variables.
Let $I$ be a set of identity literals over $\{x_0,...,x_k\}$.
Let $\varphi$ be a $\{x_0,...,x_k\}$-profile.
We say that $\varphi$ is \emph{consistent
with $I$} if $\varphi\models\bigwedge I$.
A $\UF{}$-formula $\varphi$
is a \emph{block formula} if $\varphi$ is of the type
$\exists\overline{x}\, \psi$ or $\neg\exists\overline{x}\, \psi$.
Here $\exists \overline{x}$ denotes a vector of one or more existentially quantified variables.
formulas $\exists\overline{x}\, \psi$ are called \emph{positive} blocks,
while formulas $\neg\exists\overline{x}\, \psi$ are \emph{negative} blocks.
A $\UF{}$-formula
is \emph{simple} if it is a literal or a
block formula.
Let $\tau$ be a finite relational vocabulary
with the arity bound two.
Let $x_0,...,x_k$ be distinct variable symbols.
Let $\varphi\, :=\, \exists x_1...x_k\, \psi$
be a $\UF$-formula over $\tau$. 
Let $x,y\in\{x_0,...,x_k\}$ be distinct variables.
We call $\varphi$ a \emph{$\tau$-diagram block}
if the formula $\psi$ is a conjunction
\vspace{-3pt}
$$\beta(x,y)\ \wedge\ \mathit{diff}(x_0,...,x_k)\wedge \psi_0(x_0)\
\wedge...\wedge \psi_k(x_k),$$

\vspace{-3pt}

\noindent
where $\beta(x,y)$ is a binary $\tau$-diagram in the variables $(x,y)$,
and every formula $\psi_i(x_i)$
is a conjunction of simple formulas $\psi'$ such that $\mathit{free}(\psi')\subseteq\{x_i\}$.
Furthermore, if $\chi(x)$ is a
$\tau$-diagram block with the free variable $x$, then also the formula
$\exists x\, \chi(x)$ is a $\tau$-diagram block.
%
%
%
A $\UF$-formula $\varphi$ is said to be in \emph{diagram normal form}
if for every positive block formula $\varphi'$ that occurs as a
subformula in $\varphi$, there is a $\tau$ such that 
$\varphi'$ is a $\tau$-diagram block.
\begin{lemma}\label{existentialtypelemma}
Every positive block formula is equivalent to a disjunction of diagram blocks.
\end{lemma}
\begin{proof}
Consider a block formula $\varphi(x_0) := \exists x_1...\exists x_k\, \psi$
with the free variable $x_0$.
Let $\tau$ be the set of relation symbols that occur in $\psi$.
First put $\psi$ into disjunctive normal form so that
we obtain the formula
$\chi_1\vee...\vee\chi_l$,
where each disjunct $\chi_i$ is a conjunction
of simple formulas.
We have 
$$\varphi(x_0)\ \equiv\ \varphi'(x_0)\ :=\ \exists x_1...\exists x_k\bigl(\, \chi_1\vee...\vee\chi_l\ \bigr).$$
Now distribute the existential quantifier prefix of $\varphi'(x_0)$ over
the disjunctions. We have
\begin{equation}\label{varphidashdash}
\varphi(x_0)\ \equiv\ \varphi''(x_0)\ :=\ \bigvee_{i\, \in\, \{1,...,l\}}\ \exists x_1...\exists x_k\, \chi_i.
\end{equation}
Consider an arbitrary disjunct $\exists x_1...\exists x_k\, \chi_i$
of $\varphi''(x_0)$.
Recall that $\chi_i$ is a conjunction 
simple formulas.
Let $I$ be the set of
conjuncts of $\chi_i$ that are non-trivial identity literals.
%
%
%
%
%
%
%
Let $\Pi$ be the set of all
$\{x_0,...,x_k\}$-profiles consistent with $I$.
Let $J$ be the set of conjuncts of $\chi_i$ that are not in $I$.
Thus 
\begin{equation}\label{new}
\chi_i\ \equiv\ \chi_i'\ :=\ \bigvee_{\pi\, \in\, \Pi}\bigl(\, \pi\wedge\, \bigwedge J\, \bigr).
\end{equation}
%
%
%

%
%
%
%
%
%
%
%
%
%
%

%
Write the formula $\bigwedge J$ in the form $\alpha \wedge \beta$, where $\alpha$ is a
conjunction of formulas with at most one free variable and $\beta$ is a
conjunction of literals with two free variables.
%
%
Notice that due to the uniformity condition, $\beta$ has exactly
two free variables, or, alternatively, $\beta$ is the formula $\bigwedge\emptyset = \top$.
Let $\beta_1\vee...\vee\beta_m$ be the
disjunction of $\tau$-diagrams that entail $\beta$. Thus 
\begin{equation}\label{bigwedgej}
\bigwedge J\ \equiv\ \alpha\, \wedge \bigvee\limits_{j\, \in\, [1,m]}\beta_j.
\end{equation}
Combining Equations \ref{new} and \ref{bigwedgej},
we infer that 
\begin{equation*}
\chi_i\ \equiv\ \bigvee\limits_{
\substack{\pi\, \in\, \Pi\\
j\, \in\, [1,m]}}\bigl(\, \pi\, \wedge\, \alpha\ \wedge\, \beta_j\, \bigr).
\end{equation*}
Therefore we have
\begin{equation}\label{dontcareaboutthename}
\exists x_1...\exists x_k\, \chi_i\ \equiv\ \chi_i''\ :=\ \bigvee\limits_{
\substack{\pi\, \in\, \Pi\\
j\, \in\, [1,m]}}\exists x_1...\exists x_k\bigl(\, \pi\wedge\, \alpha\, \wedge\, \beta_j \bigr).
\end{equation}
%
%
%
%
%
%

%
Consider an arbitrary disjunct $\exists x_1...\exists x_k\bigl(\, \pi\wedge\, \alpha\, \wedge\, \beta_j \bigr)$
of $\chi_i''$. The profile formula $\pi$ may 
contain unnegated literals. We want to get rid of them. 
If $\pi$ contains an unnegated identity $x=y$ as a conjunct, we get
rid of the variable $y$ altogether by erasing the conjunct $x = y$
and renaming variables in
the remaining conjuncts of $\pi\wedge\, \alpha\, \wedge\, \beta_j$.
The renaming is done such that $x_0$ is never erased.
%
%
By renaming variables in this fashion, we obtain a
formula $\pi'\wedge\, \alpha'\, \wedge\, \beta_j'$
such that 
\begin{equation}\label{chi-i}
\exists x_1...\exists x_k\bigl(\, \pi\wedge\, \alpha\, \wedge\, \beta_j \bigr)\
\equiv\ \exists y_1...\exists y_n\bigl(\, \pi'\wedge\, \alpha'\, \wedge\, \beta_j' \bigr),
\end{equation}
where $\{y_1,...,y_n\}\subseteq \{x_1,...,x_k\}$, $\pi'$ is a discriminate $\{y_1,...,y_n\}$-profile,
$\alpha'$ is a conjunction of simple formulas with at
most one free variable, and $\beta'$ is a binary $\tau$-diagram or
the formula $\top$.
%
%
%
%
%
%
%
%
%
%
By combining Equations \ref{varphidashdash},
\ref{dontcareaboutthename} and \ref{chi-i},
it is easy to see that the original formula $\varphi(x_0)$ is indeed equivalent to some
disjunction of diagram blocks.
We still need to show that positive block formulas \emph{without} free
variables are equivalent to disjunctions of diagram blocks.
Consider a formula $\exists x_0...\exists x_k\, \chi$.
We assume, w.l.o.g., that $k > 0$ and that the variable $x_0$
occurs free in $\chi$. We translate the formula $\varphi'(x_0):= \exists x_1...\exists x_k\, \chi$ in
the way described above to a disjunction $\varphi_1'\vee...\vee\varphi_n'$
of diagram blocks. We then
observe that $\exists x_0\varphi_1'\vee...\vee\exists x_0\varphi_n$ is
equivalent to the original formula $\exists x_0...\exists x_k\, \chi$.
\end{proof}
\begin{corollary}\label{typeform}
Each $\UF{}$-formula is equivalent to a formula in
diagram normal form.
\end{corollary}
\begin{proof}
By induction on the structure of formulas, using Lemma
\ref{existentialtypelemma}.
\end{proof}
Let $\tau$ be a finite relational vocabulary
with the arity bound 2.
Let $k\in\mathbb{Z}_+$, and let $x_0,...,x_k$ be distinct
variable symbols. Let $\varphi(x_0,...,x_k)\ :=\ \mathit{diff}(x_0,...,x_k) \wedge \psi$,
where $\psi$ is conjunction of $\tau$-literals
%
%
such that the following conditions hold.
\begin{enumerate}
\item
The variables of each conjunct of $\psi$ are in $\{x_0,...,x_k\}$.
\item
If $\psi$ has $R(x,y)$ or $\neg R(x,y)$ as a conjunct, where $R\in\tau$ is a
binary relation symbol and $x,y$ distinct variables,
then $x_0\in\{x,y\}$.
%
%
%
\end{enumerate}
Then we call $\varphi(x_0,...,x_k)$ a \emph{$\tau$-star formula} in the variables $(x_0,...,x_k)$.
The variable $x_0$ is called the \emph{centre variable} of $\varphi$.
Consider then a quantifier-free $\tau$-formula
$\psi(x_0,...,x_k)\ :=\ \mathit{diff}(x_0,...,x_k)\ \wedge\ \beta\ \wedge\ \alpha$
such that the following conditions are satisfied.
\begin{enumerate}
\item
The formula $\beta$ is a conjunction
$\beta_1(x_0,x_1)\wedge...\wedge\beta_k(x_0,x_k)$,
where each $\beta_i(x_0,x_i)$ is a binary $\tau$-diagram in the variables $(x_0,x_i)$.
\item
The formula $\alpha$ is a conjunction $\alpha_0(x_0)\wedge...\wedge\alpha_k(x_k)$,
where each $\alpha_i(x_i)$ is a $1$-type over $\tau$ and in the variable $x_i$.
(The variable $v_1$ is replaced by $x_i$.)
\end{enumerate}
The formula $\psi(x_0,...,x_k)$ is called a \emph{$\tau$-star type}
in the variables $(x_0,...,x_k)$. The variable $x_0$ is the \emph{centre variable}
of the $\tau$-star type.
%
%
It is straightforward to show that every $\tau$-star formula in the variables $(x_0,...,x_k)$ is
equivalent to a disjunction of $\tau$-star types in the variables $(x_0,...,x_k)$.
%
%

%
Let $\varphi(x_0,...,x_k)$ be a $\tau$-star formula
in the variables $(x_0,...,x_k)$.
Then the formula $\exists x_1...\exists x_k\varphi(x_0,...,x_k)$ is
called a \emph{$\tau$-star centre formula} of the width $k$.
Let $\psi(x_0,...,x_k)$ be a $\tau$-star type 
in the variables $(x_0,...,x_k)$.
Then the formula
$\exists x_1...\exists x_k \psi(x_0,...,x_k)$
is called a \emph{$\tau$-star centre type} of the width $k$.
The following lemma follows immediately from the fact that every $\tau$-star formula is
equivalent to a disjunction of $\tau$-star types.
\begin{lemma}\label{starcentertype}
Every $\tau$-star centre formula of the width $k$ is
equivalent to a disjunction of $\tau$-star centre types of
the width $k$.
\end{lemma}
A \emph{$2$-type} over $\tau$ is a  
maximal satisfiable set of $\tau$-literals in the variables $v_1$ and $v_2$
(equalities and negated equalities are considered to be $\tau$-literals).
If $\mathcal{T}$ is a $2$-type over $\tau$ and $x,y$ distinct variables,
we let $\mathcal{T}(x,y)$ denote the set obtained from $\mathcal{T}$ by
replacing all occurrences of $v_1$ and $v_2$ in $\mathcal{T}$
by $x$ and $y$, respectively.
%
%
%
Below we identify sets $\mathcal{T}(x,y)$
and conjunctions over them.
Let $\psi$ be a $\tau$-star type in the variables $(x_0,...,x_k)$.
Each pair $(x_0,x_i)$, where $i\in[1,k]$, is 
called a \emph{ray} of $\psi$.
Let $\mathcal{T}$ be a $2$-type over $\tau$.
We say that the ray $(x_0,x_i)$ of $\psi$ \emph{realizes} $\mathcal{T}$
if $\psi\models\mathcal{T}(x_0,x_i)$.
\begin{theorem}
Let $\tau$ be a relational vocabulary with
the arity bound $2$. 
Then $\UF{(\tau)} \leq \mathrm{FOC}^2(\tau)$.
The inclusion is strict if $\tau$ contains a binary symbol.
\end{theorem}
\begin{proof}
We have above shown that $\mathrm{FOC}^2(\tau)\not\leq \UF{(\tau)}$
if $\tau$ contains a binary relation symbol.
Therefore it suffices to show that $\UF{(\tau)}\leq \mathrm{FOC}^2(\tau)$.
The claim is established by induction on the structure of
$\mathrm{UF}_1^=$-formulas in
diagram normal form. 
We discuss the case involving quantifiers.
Let $\sigma\subseteq \tau$ and
consider a $\sigma$-diagram block $\varphi(x_0)\ :=\ \exists x_1...\exists x_k\, \psi$, where
%
%
%
%
$\psi\ :=\ \beta(x_0,x_1)\ \wedge\ \mathit{diff}(x_0,...,x_k)\ \wedge\ \psi_0(x_0)\
\wedge...\wedge\ \psi_k(x_k).$
%
%
%
%
%
%
Note that we assume that the free variable $x_0$ of $\varphi(x_0)$
occurs in the $\sigma$-diagram $\beta(x_0,x_1)$---unless $\sigma$ does not
contain binary relation symbols and thus $\beta(x_0,x_1) = \top$.
The case where $\sigma$ contains binary relation symbols and 
$x_0$ does not occur in the binary $\sigma$-diagram of $\psi$, is discussed later.
Write each formula $\psi_i(x_i)$ in a form $\delta_i(x_i)\, \wedge\, \delta_i'(x_i)$,
where $\delta_i(x_i)$ is a conjunction of the literals that occur as 
conjuncts in $\psi_i(x_i)$ and $\delta_i'(x_i)$ is the conjunction of the block formulas of $\psi_i(x_i)$.
%
%
%
We have

\vspace{-4pt}

\begin{center}
$\psi\ \equiv\ \beta(x_0,x_1)\ \wedge\
\mathit{diff}(x_0,...,x_k)\wedge \bigl(\delta_0(x_0)\wedge\delta_0'(x_0)\bigr)\
\wedge...\wedge \bigl(\delta_k(x_k)\wedge\delta_k'(x_k)\bigr).$
\end{center}

\vspace{-4pt}

\noindent
Let $P_0,...,P_k$ be fresh unary relation symbols. Consider the formula

\vspace{-4pt}

\begin{center}

$\psi'\ :=\ \beta(x_0,x_1)\ \wedge\
\mathit{diff}(x_0,...,x_k)\wedge \bigl(\delta_0(x_0)\wedge P_0(x_0)\bigr)\
\wedge...\wedge \bigl(\delta_k(x_k)\wedge P_k(x_k)\bigr).$

\end{center}

\vspace{-4pt}

\noindent
Let $\sigma'$ be the set of relation symbols in $\psi'$.
Let us consider the formula $\chi(x_0)\, :=\, \exists x_1...\exists x_k\, \psi'$.
By Lemma \ref{starcentertype}, we have
$\chi(x_0)\ \equiv\ \chi'(x_0)\, :=\, \theta_0(x_0)\vee...\vee\theta_m(x_0),$
where each $\theta_i(x_0)$ is a $\sigma'$-star centre type.
We shall next show that each $\sigma'$-star centre type can be expressed in $\mathrm{FOC}^2$.
This will conclude the argument concerning the formula $\varphi(x_0)$,
as the disjuncts $\theta_i(x_0)$ of $\chi'(x_0)$ can first be replaced
by equivalent $\mathrm{FOC}^2$-formulas, and after that, each subformula $P_i(z)$
($0\leq i\leq k$) in the resulting formula can be replaced by an $\mathrm{FOC}^2$-formula
$\delta_i''(z) \equiv \delta_i'(z)$ obtained by the induction hypothesis.
Here $z$ is either of the variables in the two-variable formula we are constructing.
If necessary, variables in $\delta_i'$ can be circulated to avoid variable capture. This way we
obtain an $\mathrm{FOC}^2$-formula
equivalent to $\varphi(x_0)$.
The notion of a star centre type was of course designed to
be expressible in $\mathrm{FOC}^2$.
Consider the $\sigma'$-star centre type $\exists x_1...\exists x_k\, \gamma$,
where $\gamma$ is the $\sigma'$-star type

\vspace{-5pt}

$$\gamma\ :=\ \mathit{diff}(x_0,...,x_k)\, 
\bigwedge\limits_{i\, \in\, \{1,...,k\}}\beta_i(x_0,x_k)
\, \wedge\bigwedge\limits_{i\, \in\, \{0,...,k\}}\alpha_i(x_i).$$
For each $2$-type $\mathcal{T}$ over $\sigma'$,
let $\#\mathcal{T}$ denote the number of rays of $\gamma$
that realize $\mathcal{T}$. Let $T$ denote
the set of all $2$-types over $\sigma'$. Define

\vspace{-6pt}

$$\gamma'(x_0)\ :=\ \bigwedge\limits_{\substack{\mathcal{T}\, \in\, T}}
                                                            \exists^{\empty\, \geq\#\mathcal{T}}y\ \ \mathcal{T}(x_0,y).$$
%
%
%
%
It is easy to see that the $\mathrm{FOC}^2$-formula $\gamma'(x_0)$ is
equivalent to the $\sigma'$-star
centre type $\exists x_1...\exists x_k\, \gamma$.
Let us then consider a $\sigma$-diagram block formula
$\theta(x_0)\, :=\, \exists x_1...\exists x_k\, \eta$, where
$\eta\ :=\ \beta(x_i,x_j)\ \wedge\ \mathit{diff}(x_0,...,x_k)\ \wedge\ \psi_0(x_0)\
\wedge...\wedge\ \psi_k(x_k),$
and $x_0\not\in\{x_i,x_j\}$. 
Let $\overline{x}$ denote a tuple 
containing exactly the variables in $\{x_0,...,x_k\}\setminus\{x_i\}$.
Consider the block formula 
$\theta'(x_i) := \exists\overline{x}\, \eta$.
In this formula, the free variable $x_i$ occurs in
the part $\beta(x_i,x_j)$ of $\eta$.
Thus, by our argument above, $\theta'(x_i)$ is equivalent
to a formula $\theta''(x_i)$ of $\mathrm{FOC}^2$.
By the induction hypothesis, there are $\mathrm{FOC}^2$-formulas
$\psi_0'(x_0)\equiv \psi_0(x_0)$
and $\psi_j'(x_j)\equiv \psi_j(x_j)$.
Let $\psi_j'(x_0)$ denote the $\mathrm{FOC}^2$-formula
obtained from $\psi_j'(x_j)$ by changing the free variable $x_j$
to $x_0$, and circulating variables, if necessary.
Let the variables used in $\psi_0'(x_0)$ and $\psi_j'(x_0)$
be $x_0$ and $x_i$.
Define $\mathcal{B}(x_i,x_0)\, :=\, \beta(x_i,x_0)\wedge\psi_j'(x_0)$.
The original formula $\theta(x_0)$ is equivalent to
the $\mathrm{FOC}^2$-formula

\vspace{-15pt}

\begin{multline*}
\psi_0'(x_0)\ \wedge\ \exists x_i\Bigl(x_i\not=x_0\wedge\theta''(x_i)\wedge\bigl(
\neg\mathcal{B}(x_i,x_0)
\vee\ \exists^{\geq 2}x_0\bigl(x_i\not=x_0\wedge \mathcal{B}(x_i,x_0)\bigr)\bigr)\Bigr).
\end{multline*}

\vspace{-6pt}

To conclude the proof, we need to discuss the case involving a block formula $\chi\ :=\ \exists x_0...\exists x_k\chi'$
that does not contain a free variable.
Assume, w.l.o.g., that $k\geq 1$.
Convert the block formula $\exists x_1...\exists x_k\chi'$
to an $\mathrm{FOC}^2$-formula $\pi(x_0)$. Thus the original formula $\chi$ is
equivalent to the $\mathrm{FOC}^2$-formula $\exists x_0\pi(x_0)$.
\end{proof}
%


\section{Undecidability of \UFC}\label{undecidabilitysection}


%
%


\newcommand{\LL}{\mathcal{L}}
\newcommand{\G}{\mathfrak{G}}
\newcommand{\NN}{\mathbb{N}}
\newcommand{\cT}{\mathbb{T}}
\newcommand{\TDUF}{\mathcal{UF}_3}
\newcommand{\ODNF}{\mathcal{NF}_1}

Since $\mathrm{FOC}^2$ and $\UF{}$ are decidable,
it is natural to ask is whether the extension of \UF{} by counting quantifiers, \UFC{}, remains decidable.
Formally, \UFC{} is obtained from \UF{} by allowing the free substitution of 
quantifiers $\exists$ by quantifiers $\exists^{\geq k},\exists^{\leq k},\exists^{=k}$.
We next show that both the general and the finite satisfiability problems of $\UFC{}$ are undecidable.
%

%
%

%
For the proofs, we use the standard tiling and periodic tiling arguments.
A \emph{tile} is a mapping $t: \{R,L,T,B\}\to C$, where $C$ is a countably infinite set of colours.
We use the subscript notation $t_X:=t(X)$ for $X\in\{R,L,T,B\}$.
Intuitively, $t_R$, $t_L$, $t_T$ and $t_B$ 
are the colors of the right edge, left edge, top edge and 
bottom edge of the tile $t$, respectively.  
Let $\str{S} := (S,H,V)$ be a structure with domain $S$ and binary relations $H$ and $V$.
Let $\cT$ be a finite nonempty set of tiles. A \emph{$\cT$-tiling} of $\str{S}$ is a function
$f:S\to\cT$ that satisfies the following conditions.
%
%
%

\vspace{-7pt}

\begin{description}
\item[$(T_H)$] For all $a,b\in S$, if $f(a)=t$, $f(b)=t'$ and $(a,b)\in H$, then $t_R=t'_L$.
\item[$(T_V)$] For all $a,b\in S$, if $f(a)=t$, $f(b)=t'$ and $(a,b)\in V$, then $t_T=t'_B$.
\end{description}

\vspace{-7pt}

%
%
%
%
\noindent
The \emph{tiling problem} for $\mathfrak{S}$ asks, given a finite nonempty set $\mathbb{T}$
of tiles, whether there exists a $\mathbb{T}$-tiling of $\mathfrak{S}$.
The standard \emph{grid} is the structure
$\mathfrak{G} := (\NN\times\NN, H,V)$,
where $H = \{\, \bigl((i,j),(i+1,j)\bigr)\, |\, i,j\in \NN\, \}$ 
and $V = \{\, \bigl((i,j),(i,j+1)\bigr)\, |\, i,j\in \NN\, \}$ are binary relations.
It is well known that the tiling problem for $\str{G}$ is $\Pi^0_1$-complete.
Let $n$ be a positive integer. Let $T := [0,n-1]\times [0,n-1]$.
An \emph{$(n\times n)$-torus} is the structure
$(T, H,V)$ such that $H = \{\, \bigl((i,j),(i+1,j)\bigr)\, |\, (i,j)\in T\ \}$
and $V = \{\, \bigl((i,j),(i,j+1)\bigr)\, |\, (i,j)\in T\, \}$,
where the sum is taken modulo $n$.
The \emph{periodic tiling problem} asks, given a 
finite nonempty set $\mathbb{T}$ of tiles,
whether there exist an $n\in\mathbb{Z}_+$
such the $(n\times n)$-torus is $\mathbb{T}$-tilable.
It is well known that the periodic tiling problem is $\Sigma_1^0$-complete.
%
%
%
%
%
%
%
%
%
%
%
%

%
We shall below define a \UFC{}-formula $\eta$ which axiomatizes a 
sufficiently rich class of grid-like structures.
In order to encode grids \emph{with $\UFC{}$-formulas}, we
employ a \emph{ternary} predicate $R$ and a unary predicate $E$. Intuitively, $E$ 
labels elements that represent the nodes of the even rows of a grid, and $R$ contains triples $(a,b,c)$ such that
$b$ is the horizontal successor of $a$ and $c$ is the vertical successor of $b$, or 
$b$ is the vertical successor of $a$ and $c$ is the horizontal successor of $b$. 
The following figure depicts an initial portion of
an infinite structure $\str{A}_\str{G}$ (over the signature $\{R,E\}$)
which is our intended encoding of the standard infinite grid
$\mathfrak{G}$.
An arrow from a node $a$ via $b$ to $c$ means that $R(a,b,c)$ holds.
%

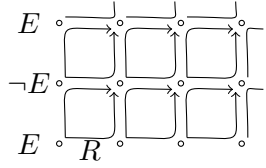
\vspace{-5pt}

\begin{center}
\begin{tikzpicture}[shorten >=1pt,->,scale=0.8]
  \tikzstyle{vertex}=[circle, draw=black,minimum size=2pt,inner sep=0pt]

\node[vertex] (a1) at (0,0) {};
\node[vertex] (a2) at (1,0) {};
\node[vertex] (a3) at (2,0) {};
\node[vertex] (a4) at (3,0) {};
\node[vertex] (a5) at (0,1) {};
\node[vertex] (a6) at (1,1) {};
\node[vertex] (a7) at (2,1) {};
\node[vertex] (a8) at (3,1) {};
\node[vertex] (a9) at (0,2) {};
\node[vertex] (a10) at (1,2) {};
\node[vertex] (a11) at (2,2) {};
\node[vertex] (a12) at (3,2) {};

\coordinate (b1) at (0.1,0.1) {};
\coordinate (b2) at (1.1,0.1) {};
\coordinate (b3) at (2.1,0.1) {};
\coordinate (b4) at (3.1,0.1) {};
\coordinate (b5) at (0.1,1.1) {};
\coordinate (b6) at (1.1,1.1) {};
\coordinate (b7) at (2.1,1.1) {};
\coordinate (b8) at (3.1,1.1) {};
\coordinate (b9) at (0.1,2.1) {};
\coordinate (b10) at (1.1,2.1) {};
\coordinate (b11) at (2.1,2.1) {};
\coordinate (b12) at (3.1,2.1) {};

\coordinate (c6) at (0.9,0.9) {};
\coordinate (c7) at (1.9,0.9) {};
\coordinate (c8) at (2.9,0.9) {};

\coordinate (c10) at (0.9,1.9) {};
\coordinate (c11) at (1.9,1.9) {};
\coordinate (c12) at (2.9,1.9) {};

\draw[->] (b1) .. controls ([xshift=28pt] b1) and ([yshift=-28pt] c6) .. (c6);
\draw[->] (b2) .. controls ([xshift=28pt] b2) and ([yshift=-28pt] c7) .. (c7);
\draw[->] (b3) .. controls ([xshift=28pt] b3) and ([yshift=-28pt] c8) .. (c8);
\draw[->] (b5) .. controls ([xshift=28pt] b5) and ([yshift=-28pt] c10) .. (c10);
\draw[->] (b6) .. controls ([xshift=28pt] b6) and ([yshift=-28pt] c11) .. (c11);
\draw[->] (b7) .. controls ([xshift=28pt] b7) and ([yshift=-28pt] c12) .. (c12);

\draw[->] (b1) .. controls ([yshift=28pt] b1) and ([xshift=-28pt] c6) .. (c6);
\draw[->] (b2) .. controls ([yshift=28pt] b2) and ([xshift=-28pt] c7) .. (c7);
\draw[->] (b3) .. controls ([yshift=28pt] b3) and ([xshift=-28pt] c8) .. (c8);
\draw[->] (b5) .. controls ([yshift=28pt] b5) and ([xshift=-28pt] c10) .. (c10);
\draw[->] (b6) .. controls ([yshift=28pt] b6) and ([xshift=-28pt] c11) .. (c11);
\draw[->] (b7) .. controls ([yshift=28pt] b7) and ([xshift=-28pt] c12) .. (c12);

\coordinate (d1) at (0.9,2.4) {};
\coordinate (d2) at (1.9,2.4) {};
\coordinate (d3) at (2.9,2.4) {};

\draw[-] (b9) .. controls ([xshift=28pt] b9) and ([yshift=-12pt] d1) .. (d1);
\draw[-] (b10) .. controls ([xshift=28pt] b10) and ([yshift=-12pt] d2) .. (d2);
\draw[-] (b11) .. controls ([xshift=28pt] b11) and ([yshift=-12pt] d3) .. (d3);

\coordinate (e1) at (3.4,0.9) {};
\coordinate (e2) at (3.4,1.9) {};

\draw[-] (b4) .. controls ([yshift=28pt] b4) and ([xshift=-12pt] e1) .. (e1);
\draw[-] (b8) .. controls ([yshift=28pt] b8) and ([xshift=-12pt] e2) .. (e2);

\coordinate [label=center:$E$] (A) at (-0.5,0);
\coordinate [label=center:$\neg E$] (A) at (-0.5,1);
\coordinate [label=center:$E$] (A) at (-0.5,2);
\coordinate [label=center:$R$] (A) at (0.5,-0.1);


	
	\end{tikzpicture}
		\end{center}
%

%
\vspace*{-5pt}
Define the formulas $\varphi_H(x,y) := \exists z \bigl(
(R(x,y,z) \vee R(z,x,y)) \wedge (E(x) \leftrightarrow E(y))\bigr)$
and $\varphi_V(x,y) := \exists z \bigl(
(R(x,y,z) \vee R(z,x,y)) \wedge (E(x) \leftrightarrow \neg E(y))\bigr)$.
Note that these are not $\UF{}$-formulas.
Let $\str{A}$ be a structure over the
vocabulary $\{R,E\}$.
We let $\str{A}^*$ be the structure over the vocabulary $\{H,V\}$
such that the $\str{A}^*$ has the same domain $A$ as $\str{A}$,
and the relation $H^{\str{A^*}}$ ($V^{\str{A^*}}$) is the set of pairs $(a,a')\in A$
such that $\str{A}\models \varphi_H[a,a']$ ($\str{A}\models\varphi_V[a,a']$).
Note that $\str{A}_\str{G}^{*}$ is the standard grid $\str{G}$.
We next define a \UFC{}-formula $\eta$ that
captures some essential properties of $\str{A}_\str{G}$.
Let $\eta$ be the conjunction of the formulas (\ref{etaa}) -- (\ref{etag}) below.
Note that the syntactic restrictions of $\UFC{}$ are indeed met.
\begin{align}\label{etaa}\vspace*{-27pt}
\exists x E(x)
\end{align}
\vspace*{-27pt}
\begin{align}\label{etab}
\forall x \exists^{=1}y \exists z (R(x,y,z) \wedge (E(x) \leftrightarrow E(y)))
\end{align}
\vspace*{-27pt}
\begin{align}\label{etac}
\forall x \exists^{=1}y \exists z (R(x,y,z) \wedge (E(x) \leftrightarrow \neg E(y)))
\end{align}
\vspace*{-27pt}
\begin{align}\label{etad}
\forall x \exists^{=1}z \exists y R(x,y,z) 
\end{align}
\vspace*{-27pt}
\begin{align}\label{etae}
\forall x\forall y\forall z (R(x,y,z) \rightarrow (E(x) \leftrightarrow \neg E(z)))
\end{align}
\vspace*{-27pt}
\begin{align}\label{etaf}
\forall x \exists^{=1}y \exists z (((E(x) \leftrightarrow  E(y)) \wedge (R(z,x,y) \vee R(x,y,z)))
\end{align}
\vspace*{-27pt}
\begin{align}\label{etag}
\forall x \exists^{=1}y \exists z (((E(x) \leftrightarrow \neg E(y)) \wedge (R(z,x,y) \vee R(x,y,z)))
\end{align}

We claim that $\eta$ has the following properties.

\vspace*{-6pt}

\begin{enumerate}\itemsep0pt
\item[(i)] There exists a model $\str{A} \models \eta$ such that $\str{A}^* = \str{G}$.
\item[(ii)]  For every model $\str{A} \models \eta$, there is a
homomorphism from $\str{G}$ to $\str{A}^*$.
\end{enumerate}

\vspace*{-6pt}

Assume we can show that $\eta$ indeed has the above properties.
Let $\mathbb{T}$ be an arbitrary input to the tiling problem,
and let $\mathcal{P}_{\mathbb{T}}\, :=\, \{\, P_t\, |\, t\in\mathbb{T}\, \}$ be a set of 
fresh unary predicate symbols. 
Construct a $\UFC{}$-formula $\varphi_{\mathbb{T}} := \psi_0\wedge\psi_H\wedge\psi_V$
over the vocabulary $\{R,E\}\cup\mathcal{P}_{\mathbb{T}}$ as follows.
\begin{enumerate}
\item
$\psi_0$ states that each point of the model is in the interpretation of
exactly one predicate symbol $P_t$, $t\in\mathbb{T}$.
\item
$\psi_H\, \equiv\, \forall x\forall y\bigwedge_{t,t'\in\mathbb{T},\, t_{R}\not=t_{L}'}\neg
\bigl(\varphi_H(x,y)\wedge P_t(x)\wedge P_{t'}(y)\bigr)$. Note that the right hand side here
is not a $\UFC$-formula, but it can easily be modified so that the resulting formula is.
\item
$\psi_V\equiv \forall x\forall y\bigwedge_{t,t'\in\mathbb{T},\, t_{T}\not=t_{B}'}\neg
\bigl(\varphi_V(x,y)\wedge P_t(x)\wedge P_{t'}(y)\bigr)$.
\end{enumerate}
It is easy to see that $\eta \wedge \varphi_{\mathbb{T}}$
has a model iff there exists a $\mathbb{T}$-tiling of $\str{G}$.
\begin{claim}
The formula $\eta$ has the properties (i) and (ii).
\end{claim}
\begin{proof}
For (i), we take as $\str{A}$ the structure $\str{A}_\str{G}$.
Let us consider the property (ii). Let $\str{A}$ be a
model such that $\str{A} \models \eta$. We will show how to construct
a homomorphism $h$ from the standard  grid $\str{G}$ to $\str{A}^*$.
Let us first define the embedding of the first two rows of the grid.
Let $a' \in A$ be an element witnessing the conjunct (\ref{etaa}). Let $h(0,0)=a'$. 
Let $b',c' \in A$ be elements such that $\str{A} \models R[a',b',c'] \wedge (E[a'] \leftrightarrow \neg E[b'])$,
guaranteed to exist by (\ref{etac}). Let $h(0,1)=b'$.
We observe that $V[h(0,0),h(1,0)]$ holds in $\str{A}^*$, as required.
We also have $\str{A}\models E[a']\wedge\neg E[b']$.
Assume, for the sake of induction, that we have defined
$h(k,0)=a$ and $h(k,1) = b$ for some $k \ge 0$. Assume that
$\str{A}^*\models H[h(m',0),h(m'+1,0)]$ and $\str{A}^*\models V[h(m,0),h(m,1)]$
for all $m'\in\{0,...,k-1\}$ and $m\in\{0,...,k\}$.
Assume also that $\str{A}\models E[h(m,0)]\wedge\neg E[h(m,1)]$
for all $m\in\{0,...,k\}$.
Let $c,d \in A$ be elements such that
$\str{A} \models R[a,c,d] \wedge (E[a] \leftrightarrow E[c])$,
%
%
guaranteed to exist by (\ref{etab}).
%
%
Choose $h(k+1, 0)=c$ and $h(k+1, 1)=d$.
%
%
%
We have $H[h(k,0), h(k+1,0)]$ in $\str{A}^*$.
Let $c',d'\in A$ be elements such that
$\str{A} \models R[a,c',d'] \wedge (E[a] \leftrightarrow \neg E[c'])$,
guaranteed to exist by (\ref{etac}).
We have $d' = d$ by (\ref{etad}).
Thus $\str{A}\models R[a,c',d] \wedge (E[a] \leftrightarrow \neg E[c'])$.
By (\ref{etag}), we conclude that $c' = b$.
By (\ref{etae}), we have $\str{A}\models (E[a] \leftrightarrow \neg E[d])$,
and thus $\str{A}\models (E[d] \leftrightarrow E[c'])$.
Hence we have $\str{A}\models R[a,b,d] \wedge (E[b] \leftrightarrow E[d])$,
whence $H[h(k,1),h(k+1,1)]$ holds in $\str{A}^*$.
We still need to show that $\str{A}^*\models V[h(k+1,0),h(k+1,1)]$, i.e.,
$\str{A}^*\models V[c,d]$.
We already know that $\str{A}$ satisfies $E(a) \leftrightarrow E(c)$ and
$E(a) \leftrightarrow \neg E(d)$.
Thus $\str{A}\models E(c)\leftrightarrow \neg E(d)$. 
Since we also know that $\str{A}\models R[a,c,d]$,
we conclude that $\str{A}^*\models V[c,d]$.
Since $\str{A}\models E[a] \wedge (E[a]\leftrightarrow E[c])
\wedge (E[a]\leftrightarrow \neg E[d])$,
we have $\str{A}\models E[h(k+1,0)]\wedge \neg E[h(k+1,1)]$.
%

%
%

%
We have defined $h$ for the first two rows of $\str{G}$.
Assume, for the sake of induction, that
we have defined $h$ for the first $l\geq 2$ rows of the grid,
and that the homomorphism conditions are satisfied.
Assume also that for all $m\in\mathbb{N}$ and $m'\in \{0,...,l\}$,
we have $\str{A}\models E[h(m,m')]$ iff $m'$ is even.
We extend the definition of $h$ to the $(l+1)$-st row.
As above, we proceed by induction on the columns.
Let $h(0,l)=a$. Let $b,c \in A$ be elements such that $\str{A}
\models R[a,b,c] \wedge (E[a] \leftrightarrow \neg E[b])$,
guaranteed to exist by (\ref{etac}).
Let $h(0,l+1)=b$. We may assume, by symmety,
that $\str{A}\models E[h(0,l)]$.
We have $\str{A}^*\models V[h(0,l),h(0,l+1)]$
and $\str{A}\models \neg E[h(0,l+1)]$.
Assume, for the sake of induction, that there exists a $k$ such that we 
have defined $h(m,l+1)$ for all $m\leq k$. Assume that 
$\str{A}^*\models H[h(m',l+1),h(m'+1,l+1)]$ for all $m'\in\{0,...,k-1\}$ and
$\str{A}^*\models V[h(m,l),h(m,l+1)]$ for all $m\in\{0,...,k\}$.
Assume also that $\str{A}\models \neg E[h(m,l+1)]$ for all $m\in\{0,...,k\}$.
Let $b, c \in A$ be elements such that
$\str{A} \models R[a,b,c] \wedge (E[a] \leftrightarrow E[b])$,
guaranteed to exist by (\ref{etab}).
Define $h(k+1, l+1)=c$.
The arguments concerning the homomorphism conditions (and 
the predicate $E$) are similar to the arguments concerning the first two rows,
but involve also the condition (\ref{etaf}). The
straightforward details are left to the reader.
\end{proof}
Thus the satisfiability problem of $\UFC{}$ is $\Pi_1^0$-hard. Since $\UFC{}$ is a fragment
of first-order logic, the following theorem holds.
\begin{theorem}\label{theoremundec}
The satisfiability problem of\, $\UFC{}$ is $\Pi_1^0$-complete.
\end{theorem}
The above argument leading to Theorem \ref{theoremundec} can be used
with minor modifications in order to show $\Sigma_1^0$-completeness of the
finite satisfiability problem of $\UFC{}$ using the periodic tiling problem.
\begin{theorem}\label{undecfinite}
The finite satisfiability problem of\, $\UFC{}$ is $\Sigma_1^0$-complete.
\end{theorem}
\begin{proof}
We use the same formulas $\eta$ and $\varphi_{\mathbb{T}}$
as in the proof of Theorem \ref{theoremundec}.
We claim that $\eta$ has the following properties.
\begin{enumerate}\itemsep0pt
\item[(iii)] For every positive integer $n$,
there is a structure $\str{A} \models \eta$ such that $\str{A}^*$ is the $(2n\times 2n)$-torus.
(We use the factor $2$ here in order to deal with the predicate $E$ appropriately.)
\item[(iv)]  For every finite model $\str{A} \models \eta$, there exists some
$(n\times n)$-torus $\str{T}$ such that there is a homomorphism
from $\str{T}$ to $\str{A}^*$.
\end{enumerate}
Assume we can show that $\eta$ satisfies (iii) and (iv).
It is then easy to show that for a nonempty finite set $\mathbb{T}$ of tiles,
the formula $\eta\wedge\varphi_{\mathbb{T}}$ has a finite model
iff there exists some $n\in\mathbb{Z}_+$ such that the $(n\times n)$-torus is $\mathbb{T}$-tilable.
Thus the finite satisifiability problem for $\UFC{}$ is $\Sigma_1^0$-hard.
Since the finite satisfiability problem of first-order logic is in $\Sigma_1^0$,
the theorem holds.
%
%
%
%

%
We then sketch a proof that $\eta$ indeed has the properties (iii) and (iv).
For each positive integer $n$, we let $\str{T}_n$ denote the $(n\times n)$-torus.
For (iii), we define $\str{A}$ to be the natural quotient of $\str{A}_{\str{G}}$
with $2n\times 2n$ elements such that $\str{A}^* = \str{T}_{2n}$.
To prove (iv), let $\str{A}$ be a finite structure such that $\str{A}\models\eta$.
We first construct a homomorphism
$h$ from the standard  grid $\str{G}$ to $\str{A}^*$
precisely as in the proof of Theorem \ref{theoremundec}.
Since $\str{A}$ is finite,
there exist some $i,j$, $ i < j$, such that we have $h(i,0) = h(j,0)$.
We can show by induction, using formula
(\ref{etag}), that $h(i,s)=h(j,s)$ for all $s \in \N$.
Now, for each $s \in \mathbb{N}$, let $\pi_s$
denote the tuple
$$\bigl(h(i,s), h(i+1, s), \ldots, h(j-1,s)\bigr).$$
Finitness of $\str{A}$ guarantees that there exist $k,l$, $k < l$
such that $\pi_k$ = $\pi_l$.
Let $h': [0,j-i-1] \times [0, l-k-1] \rightarrow A$ be the function
defined so that $h'(x,y) = h(i+x,k+y)$.
The arguments given above imply that $h'$ is a homomorphism from
the (not necessarily square) $(j-i) \times (l-k)$ torus to $\str{A}^*$.
Define $t\, :=\, \mathit{lcm}(j-i,\,  l-k,\,  2)$,
where $\mathit{lcm}$ denotes the least common multiple operation.
Let $m := \frac{t}{2}$.
Using $h'$, it is easy to define an embedding from the
square $(2m \times 2m)$-torus to $\str{A}^*$.
\end{proof}
%


\bibliography{uodf}




\end{document}